\documentclass[12pt,en]{elegantpaper}
\usepackage{esint}
\usepackage{extarrows}
\usepackage{mathrsfs}
\numberwithin{equation}{section}
\allowdisplaybreaks[4]
\everymath{\displaystyle}

\title{\textbf{Classification of Solutions to a Critical Fourth-Order\\ Equation on Complete Non-compact Manifolds}}

\author{Huabin Li \and Tian Wu \and Xiao Zhou}
\date{}

\begin{document}

\maketitle

\begin{abstract}
    \setlength{\abovedisplayskip}{2pt}
    \setlength{\belowdisplayskip}{2pt}
    \setlength{\abovedisplayshortskip}{0pt}
    \setlength{\belowdisplayshortskip}{0pt}
    \hspace{1em}
    In this paper, we study the critical biharmonic equation
    \[
    \Delta ^2 u=u^{\frac{n+4}{n-4}},
    \]
    on a complete, connected and non-compact Riemannian manifold \((M^n,g)\) of dimension \(n\geqslant5\) with non-negative Ricci curvature. We establish an optimal pointwise second-order derivative estimate by Bernstein's method and the continuity method. Using the method of invariant tensors, we derive a differential inequality that leads to a classification result: if there exists a positive finite energy solution, then the manifold is isometric to the Euclidean space \(\mathbb{R}^n\), and
    \begin{equation*}
        u(x)
        =
        \frac{
        \left[\lambda^2 n(n-4)(n^2-4)\right]^{\frac{n-4}{8}}
        }
        {\left(1+\lambda|x-x_0|^2\right)^{\frac{n-4}{2}}},
        \qquad
        \lambda>0,\quad x_0\in\mathbb{R}^n.
    \end{equation*}
\end{abstract}

\tableofcontents
\section{Introduction}

Let \((M^n,g)\) be a complete, connected and non-compact Riemannian manifold of dimension \(n\geqslant5\) with non-negative Ricci curvature. Let \(\nabla\) and \(\Delta\) be the Levi-Civita connection and the Laplace-Beltrami operator w.r.t. the metric \(g\). Denote the Ricci curvature as \(\mathrm{Ric}\). We consider
\begin{equation}
    \Delta^2 u=f(u),
    \qquad
    u\in \mathcal{C}^5(M,\mathbb R_+),
    \qquad
    f\in \mathcal{C}^1(\mathbb R_+,\mathbb R_+),
    \tag{1.1}
    \label{eq}
\end{equation}
where $\mathbb{R}_+$ denotes the set of positive real numbers. The critical nonlinearity is $f(u)=u^{\frac{n+4}{n-4}}$.

One of the backgrounds of the biharmonic equation is seeking the best constant in the Sobolev embedding $\mathcal{H}^2(\mathbb{R}^n)\hookrightarrow \mathcal{L}^{\frac{2n}{n-4}}(\mathbb{R}^n)$, and by the calculus of variations, we obtain that the extremal functions satisfy the equation
\[
    \Delta^2 u=|u|^{\frac{8}{n-4}}u.
\]
Another background is that when $f(u)=u^p$, the biharmonic equation can be considered as a special case (when \(q=1\)) of the Lane-Emden system
\begin{equation*}
    \begin{cases}
        -\Delta u = v^q, & x \in \Omega,\\
        -\Delta v = u^p, & x \in \Omega,
    \end{cases}
\end{equation*}
where $p\geqslant q\geqslant1$. In this case, Ma-Wu-Wu \cite{MWW25} proved non-existence of positive superharmonic solution to \(\Delta^2 u=u^{p}\) in the subcritical case \(1<p<\frac{n+4}{n-4}\) with \(n\geqslant 5\).

In the critical case, with the help of the invariant tensor technique, we obtain the following classification theorems. Our first result concerns the case $n=5$ and is stated as follows.
\begin{theorem}
Let $(M^5,g)$ be a complete, connected, non-compact Riemannian manifold
of dimension $n=5$ satisfying $\mathrm{Ric}\geqslant 0$. Then there exists a positive smooth function satisfying
\[
    \Delta ^2 u=u^9,\qquad \limsup_{x\to\infty}\Delta u(x)\leqslant0,
\]
if and only if $(M^5,g)$ is isometric to $(\mathbb{R}^5,g_E)$. Moreover, $u$ is of the form
\[
    u(x)
    =
    \lambda^{\frac14}105^{\frac18}
    \left(1+\lambda |x-x_0|^2\right)^{-\frac12},
    \qquad
    \lambda\in\mathbb{R}_+,\quad x_0\in\mathbb{R}^5.
\]
\end{theorem}
For a general dimension $n\geqslant5$, we impose an additional integrability condition.
\begin{theorem}
Let $(M^n,g)$ be a complete, connected, non-compact Riemannian manifold of dimension $n\geqslant 5$ satisfying $\mathrm{Ric}\geqslant 0$. Then there exists $u\in \mathcal{C}^\infty(M,\mathbb R_+)\cap\mathcal{L}^{\frac{2n}{n-4}}(M)$ satisfying $\Delta ^2 u=u^{\frac{n+4}{n-4}}$ if and only if $(M^n,g)$ is isometric to $(\mathbb{R}^n,g_E)$. Moreover, $u$ is of the form
\begin{equation*}
    u(x)
    =
    \frac{
    \left[\lambda^2 n(n-4)(n^2-4)\right]^{\frac{n-4}{8}}
    }
    {\left(1+\lambda|x-x_0|^2\right)^{\frac{n-4}{2}}},
    \qquad
    \lambda\in\mathbb{R}_+,\quad x_0\in\mathbb{R}^n.
\end{equation*}
\end{theorem}
Furthermore, the integrability condition can be relaxed to the following weaker assumption: for some $p\in M$ and large $R$,
\begin{equation*}
    \int_{B(p, R)}
    u^{\frac{2n}{n-4}}
    =
    O(R^{\nu_n}),
    \qquad
    \nu_n
    =
    \begin{cases}
        2,\qquad n=5\ or\ 6,\\
        \frac{n}{n-1},\qquad n\geqslant7.
    \end{cases}
\end{equation*}

Related second-order semilinear equations in complete, connected, and non-compact Riemannian manifolds with non-negative Ricci curvature have been widely studied. Gidas-Spruck \cite{GS81} proved non-existence of positive solution to \(-\Delta u=u^p\) in the subcritical case \(1<p<\frac{n+2}{n-2}\) with \(n\geqslant 3\). The critical case \(p=\frac{n+2}{n-2}\) is more difficult. Caffarelli-Gidas-Spruck \cite{CGS89} (also see Chen-Li \cite{CL91}) classified positive solutions via the moving plane method in \(\mathbb R^n\). The rigidity result in the critical case, that is, the manifold must be \(\mathbb R^n\), was investigated under some assumptions of decay at infinity or integrable condition, such as Fogagnolo-Malchiodi-Mazzieri \cite{FMM23}, Catino-Monticelli \cite{CM22}. Recently, to remove assumptions, some works succeeded in lower dimensions, such as Catino-Monticelli-Roncoroni \cite{CMR23}, Ou \cite{Ou25}, V\'etois \cite{Vet24} and Sun-Wang \cite{SW25}.

Let us focus on the fourth-order case. Lin \cite{Lin98} classified exponentially integrable solutions with a growth condition, namely $e^{4u}\in \mathcal{L}^1(\mathbb R^4)$ and $|u(x)| = o(|x|^2)$ at $\infty$, to \(\Delta^2u=\mathrm 6e^{4u}\) in \(\mathbb R^4\) and positive solutions to \(\Delta^2u=u^p\) in \(\mathbb R^n\) with \(n\geqslant5\) and \(1<p\leqslant\frac{n+4}{n-4}\) by the moving plane method. However, it doesn't work in general manifolds. It's noteworthy that corresponding problems on closed manifolds with a positive lower Ricci curvature bound have been studied recently. We recommend readers to see V\'etois \cite{Vet24}, Li-Wei \cite{LW25}, Case \cite{Cas24}, Ma-Wu-Zhou \cite{MWZ25a,MWZ25b}.

In Theorem 1.1 and 1.2, we extend Lin's theorem in the critical case to complete, connected, and non-compact manifolds with non-negative Ricci curvature. Unlike backgrounds from conformal geometry in \cite{Vet24,LW25}, we establish differential identities using the invariant tensor technique. This technique was introduced by Ma-Wu \cite{MW24} to study second-order semi-linear equations on manifolds. Ma-Wu-Zhou \cite{MWZ25a} extended results of V\'etois and Li-Wei on closed manifolds via this technique.

Throughout this paper, \(\varepsilon\) and \(C\) denote positive constants independent of the radius \(R\). Besides, \(\varepsilon\) is always sufficiently small. Since $(M^n,g)$ is complete and $\mathrm{Ric}\geqslant0$, for any $p\in M$ and $R\in\mathbb{R}_+$, there exists a standard cut-off function $\eta\in \mathcal{C}_c^\infty(M)$ constructed via the heat kernel such that
\[
    0\leqslant\eta\leqslant1,
    \qquad
    \operatorname{supp}(\eta)\subset B(p,2R),
    \qquad
    \eta\equiv1\quad\text{on }B(p,R),
\]
and
\begin{equation}
    |\nabla\eta|\leqslant\frac{C}{R},
    \qquad
    |\Delta\eta|\leqslant\frac{C}{R^2},
    \tag{1.2}
    \label{eta}
\end{equation}
where $C$ depends only on $n$. We write $B_R:=B(p,R)$.

In local coordinates, we denote the metric tensor as \(g_{ij}\), and the Ricci curvature tensor as \(R_{ij}\). Denote the inner product w.r.t. \(g\) as \(\langle\cdot,\cdot\rangle\). We employ the Einstein summation convention to lower and raise the indices using the metric tensor \(g_{ij}\) and its inverse \(g^{ij}\). All covariant derivatives are denoted with a comma, such as \(_{,i}\), except for acting on our solution \(u\) and the cutoff function \(\eta\).

Regarding the organization of this paper, in Section 1 we prove that the superharmonicity of the solution can be derived from two different conditions, respectively. In Section 2, we establish a prior estimate of second derivatives via Bernstein's technique and the continuity method. In Section 3, the crucial differential inequality is obtained by the invariant tensor technique. In Section 4, we establish integral estimates to prove Theorem 1.1 and 1.2.

The purpose of this paper is to classify positive solutions of the critical equation under the assumptions that $\mathrm{Ric}\geqslant0$ and $\Delta u\leqslant 0$. When the manifold is $\mathbb{R}^n$, Fazly-Wei-Xu \cite{FWX15} showed that if
\[
    \Delta^2u\geqslant |x|^a u^p,\qquad u\in \mathcal{C}^4(\mathbb R^n,\mathbb R_+),\quad p>1,\quad a\geqslant0,
\]
then $\Delta u\leqslant0$. When $M$ is a general complete, connected and non-compact manifold, if $\Delta^2 u\geqslant0$, then $\Delta u\leqslant0$ follows from a weaker condition
\begin{equation}
    \limsup_{d(p,x)\to+\infty}\Delta u(x)\leqslant0.
    \tag{1.3}
    \label{Delta u}
\end{equation}
Let \(v=\Delta u\), then \(\Delta v\geqslant0\). Suppose that there exists \(q\in M\) such that \(v(q)>0\). By the assumption, there exists a compact set \(K\) such that \(v(x)<v(q)\) for all \(x\notin K\). Hence \(v\) has a global maximum point \(x_0\in K\), thus the set \(\{v=v(x_0)\}\) is non-empty and closed. On the other hand, the strong maximum principle and \(\Delta v\geqslant0\) give that \(\{v=v(x_0)\}\) is open. Therefore, by the connectedness of the manifold, \(v\equiv v(x_0)>0\), which contradicts \eqref{Delta u}. The superharmonicity condition can also be deduced from the following integrability assumption.
\begin{proposition}
Let $(M^n,g)$ be a complete, connected, non-compact Riemannian manifold of dimension $n\geqslant5$ with non-negative Ricci curvature. Let $u$ be a positive $\mathcal{C}^4$ function on $M$ such that $0\leqslant\Delta^2u\leqslant Cu^{\frac{n+4}{n-4}}$ and for some $p\in M$ and large $R$,
\begin{equation}
    \int_{B_R}
    u^{\frac{2n}{n-4}}
    =
    O(R^2).
    \tag{1.4}
    \label{integrability}
\end{equation}
Then $\Delta u\leqslant0$.
\end{proposition}
\begin{proof}
Let $v=\Delta u$ and $v_+=\max\{v,0\}$, then $0\leqslant\Delta v\leqslant u^{\frac{n+4}{n-4}}$. We first derive a Caccioppoli-type estimate. For $p\in M$ and large $R$, choose the cut-off function $\eta$ associated with $\frac R2$ and take $\eta^6v_+$ as a test function. Integration by parts gives
\begin{align*}
0\leqslant{}
\int_M\eta^6v_+\Delta v
&=
-\int_M\langle\nabla(\eta^6v_+),\nabla v\rangle
\nonumber\\
&=
-\int_M\eta^6\langle\nabla v_+,\nabla v\rangle
-
6\int_M\eta^5 v_+\langle\nabla\eta,\nabla v\rangle
\nonumber\\
&=
-\int_M\eta^6|\nabla v_+|^2
-
6\int_M\eta^5 v_+\langle\nabla\eta,\nabla v_+\rangle.
\end{align*}
Thus
\begin{align*}
\int_M\eta^6|\nabla v_+|^2
&\leqslant{}
6\int_{B_R\setminus B_{R/2}}\eta^5 v_+|\nabla\eta||\nabla v_+|
\nonumber\\
&\leqslant{}
\frac12\int_M\eta^6|\nabla v_+|^2
+
C\int_M\eta^4v_+^2|\nabla \eta|^2.
\tag{1.5}
\label{v 1}
\end{align*}
Then by \eqref{eta},
\begin{align*}
\int_M\eta^6|\nabla v_+|^2
&\leqslant{}
C\int_M\eta^4v_+^2|\nabla \eta|^2
\leqslant{}
CR^{-2}\int_M(\Delta u)^2\eta^4.
\tag{1.6}
\label{v 2}
\end{align*}
Again integration by parts and \eqref{eta} give
\begin{align*}
\int_M(\Delta u)^2\eta^4
&\leqslant{}
-\int_M\langle\nabla \Delta u,\nabla u\rangle\eta^4
+
CR^{-1}
\int_M|\Delta u||\nabla u|\eta^3
\nonumber\\
&=
\int_Mu\Delta^2 u\eta^4
+
\int_Mu\langle\nabla \Delta u,\nabla \eta^4\rangle
+
CR^{-1}
\int_M|\Delta u||\nabla u|\eta^3
\nonumber\\
&=
\int_Mu\Delta^2 u\eta^4
-
\int_Mu\Delta u\Delta \eta^4
-
\int_M\Delta u\langle\nabla u,\nabla \eta^4\rangle
+
CR^{-1}
\int_M|\Delta u||\nabla u|\eta^3
\nonumber\\
&\leqslant{}
\int_Mu\Delta^2 u\eta^4
+
CR^{-1}
\int_M|\Delta u||\nabla u|\eta^3
+
CR^{-2}
\int_Mu|\Delta u|\eta^2
\nonumber\\
&\leqslant{}
C\int_{B_R}u^{\frac{2n}{n-4}}
+
\frac12\int_M(\Delta u)^2\eta^4
+
CR^{-2}
\int_M|\nabla u|^2\eta^2
+
CR^{-4}
\int_{B_R}u^2.
\end{align*}
Since
\begin{align*}
\int_M|\nabla u|^2\eta^2
&\leqslant{}
\int_Mu|\Delta u|\eta^2
+
CR^{-1}
\int_Mu|\nabla u|\eta
\nonumber\\
&\leqslant{}
\int_Mu|\Delta u|\eta^2
+
\frac12\int_M|\nabla u|^2\eta^2
+
CR^{-2}
\int_{B_R}u^2,
\end{align*}
therefore,
\begin{align*}
\int_M(\Delta u)^2\eta^4
&\leqslant{}
C\int_{B_R}u^{\frac{2n}{n-4}}
+
CR^{-2}
\int_Mu|\Delta u|\eta^2
+
CR^{-4}
\int_{B_R}u^2
\nonumber\\
&\leqslant{}
C\int_{B_R}u^{\frac{2n}{n-4}}
+
\frac12\int_M(\Delta u)^2\eta^4
+
CR^{-4}
\int_{B_R}u^2.
\end{align*}
Then by \eqref{integrability}, \eqref{v 2}, the Bishop--Gromov volume comparison theorem and Hölder's inequality,
\begin{align*}
w_R
:=
\int_{B_{R/2}}|\nabla v_+|^2
&\leqslant{}
\int_M\eta^6|\nabla v_+|^2
\leqslant{}
CR^{-2}\int_M(\Delta u)^2\eta^4
\nonumber\\
&\leqslant{}
CR^{-2}
\int_{B_R}u^{\frac{2n}{n-4}}
+
CR^{-6}
\int_{B_R}u^2
\nonumber\\
&\leqslant{}
CR^{-2}
\int_{B_R}u^{\frac{2n}{n-4}}
+
CR^{-2}
\left(
\int_{B_R}u^{\frac{2n}{n-4}}
\right)^{\frac{n-4}{n}}
\leqslant{}C.
\end{align*}
Now we use the Caccioppoli trick. Since $w_R$ is bounded and increasing with respect to $R$, the limit exists as $R\to+\infty$, which is denoted by $w$. Then we have
\[
    w=\int_M|\nabla v_+|^2\leqslant{}C.
\]
By \eqref{v 1}, for any $A\in\mathbb{R}_+$,
\begin{align*}
w_R
&\leqslant{}
\int_M\eta^6|\nabla v_+|^2
\leqslant{}
6\int_{B_R\setminus B_{R/2}}\eta^5 v_+|\nabla\eta||\nabla v_+|
\nonumber\\
&\leqslant{}
A\int_{B_R\setminus B_{R/2}}\eta^6|\nabla v_+|^2
+
CA^{-1}R^{-2}\int_M(\Delta u)^2\eta^4
\leqslant{}
A\left(w-w_R\right)
+
CA^{-1}.
\end{align*}
Then letting $R\to+\infty$ gives $w\leqslant{}CA^{-1}$. Finally, letting $A\to+\infty$ gives
\begin{equation}
    w=\int_{\{v>0\}}|\nabla v|^2=0.
    \tag{1.7}
    \label{v 3}
\end{equation}
Suppose that there exists \(q\in M\) such that $v(q)>0$. Consider the set $\{v=v(q)\}$, which is obviously non-empty and closed. If $v(x)=v(q)>0$, then we can choose a sufficiently small good coordinate neighborhood $U_x\subseteq\{v>0\}$. Thus from \eqref{v 3}, we obtain that $\nabla v\equiv0$ on $U_x$. Hence $U_x\subseteq\{v=v(q)\}$ and then $\{v=v(q)\}$ is open. Consequently, by the connectedness, $v\equiv v(q)>0$. Take $\eta^2u^{\frac{n+4}{n-4}}$ as a test function and integration by parts gives
\begin{align*}
0\leqslant{}
\int_M\eta^2u^{\frac{n+4}{n-4}}\Delta u
=
-\frac{n+4}{n-4}\int_M\eta^2u^{\frac{8}{n-4}}|\nabla u|^2
-2\int_M\eta u^{\frac{n+4}{n-4}}\langle\nabla u,\nabla \eta\rangle.
\end{align*}
Thus
\begin{align*}
\int_M\eta^2u^{\frac{8}{n-4}}|\nabla u|^2
&\leqslant{}
CR^{-1}
\int_{B_R\setminus B_{R/2}}\eta u^{\frac{n+4}{n-4}}|\nabla u|
\nonumber\\
&\leqslant{}
\frac12\int_M\eta^2u^{\frac{8}{n-4}}|\nabla u|^2
+
CR^{-2}
\int_{B_R}u^{\frac{2n}{n-4}}.
\tag{1.8}
\label{v 4}
\end{align*}
Hence by \eqref{integrability},
\begin{align*}
\int_{B_{R/2}}u^{\frac{8}{n-4}}|\nabla u|^2
\leqslant{}
\int_M\eta^2u^{\frac{8}{n-4}}|\nabla u|^2
&\leqslant{}
CR^{-2}
\int_{B_R}u^{\frac{2n}{n-4}}
\leqslant{}C.
\end{align*}
Again by the Caccioppoli trick and \eqref{v 4}, we have
\[
    \int_Mu^{\frac{8}{n-4}}|\nabla u|^2=0.
\]
This implies that $u$ is constant, contradicting $\Delta u>0$.
\end{proof}
\section{Second-Order Derivative Estimates}

\subsection{Preliminaries}

Let $(M^n,g)$ be a complete, connected, non-compact Riemannian manifold,
$n\geqslant 5$, satisfying $\mathrm{Ric}\geqslant 0$, and let
\[
    u\in \mathcal{C}^4(M,\mathbb R_+),
    \qquad
    \Delta^2u\geqslant 0.
\]
For $a\geqslant 0$, introduce the auxiliary quantity
\begin{equation*}
    Z_a
    :=
    \frac{\Delta u}{u}
    +
    a\frac{|\nabla u|^2}{u^2},
    \qquad
    X_a
    :=
    Z_a+b_na u^{\frac{4}{n-4}},
    \qquad
    b_n
    :=
    \frac{n-4}{2}
    \sqrt{\frac{n(n-4)}{n^2-4}}.
\end{equation*}
A direct differentiation gives
\begin{align*}
    \nabla Z_a
    ={}&
    -\frac{\Delta u}{u^2}\nabla u
    +\frac{1}{u}\nabla\Delta u
    -2a\frac{|\nabla u|^2}{u^3}\nabla u
    +2a\frac{1}{u^2}\nabla^2 u\cdot\nabla u.
\end{align*}
Differentiating once more, we obtain
\begin{align*}
    \Delta Z_a
    ={}&
    \frac{\Delta^2u}{u}
    -\frac{2}{u^2}
        \langle\nabla\Delta u,\nabla u\rangle
    +\Delta u
        \left(
        \frac{2|\nabla u|^2}{u^3}
        -\frac{\Delta u}{u^2}
        \right)
    \nonumber\\
    &+
    \frac{a}{u^2}\Delta|\nabla u|^2
    -\frac{4a}{u^3}
        \langle\nabla|\nabla u|^2,\nabla u\rangle
    +a|\nabla u|^2
        \left(
        \frac{6|\nabla u|^2}{u^4}
        -\frac{2\Delta u}{u^3}
        \right).
\end{align*}
Using the Bochner formula, we obtain
\begin{align*}
    \Delta Z_a
    ={}&
    \frac{\Delta^2u}{u}
    -\frac{(\Delta u)^2}{u^2}
    +\frac{2a-2}{u^2}
        \langle\nabla\Delta u,\nabla u\rangle
    +\frac{2-2a}{u^3}
        \Delta u|\nabla u|^2
    \nonumber\\
    &+
    \frac{2a}{u^2}|\nabla^2 u|^2
    +\frac{6a}{u^4}|\nabla u|^4
    -\frac{8a}{u^3}
        \langle\nabla^2 u,\nabla u\otimes\nabla u\rangle
    +\frac{2a}{u^2}\mathrm{Ric}(\nabla u,\nabla u).
\end{align*}
Consequently,
\begin{align*}
u^{2a-2}
\operatorname{div}
\left(u^{2-2a}\nabla Z_a\right)
={}&
\frac{\Delta^2u}{u}
-\frac{(\Delta u)^2}{u^2}
+\frac{2a(2a+1)}{u^4}|\nabla u|^4
+\frac{2a}{u^2}|\nabla^2 u|^2
\nonumber\\
&-
\frac{4a(a+1)}{u^3}
\langle\nabla^2 u,\nabla u\otimes\nabla u\rangle
+\frac{2a}{u^2}\mathrm{Ric}(\nabla u,\nabla u).
\end{align*}
Completing the square gives
\begin{align*}
u^{2a-2}
\operatorname{div}
\left(u^{2-2a}\nabla Z_a\right)
={}&
2a
\left|
\frac{\nabla^2 u}{u}
-(1+a)\frac{\nabla u\otimes\nabla u}{u^2}
-\frac1n
\left(
\frac{\Delta u}{u}
-(1+a)\frac{|\nabla u|^2}{u^2}
\right)g
\right|^2
\nonumber\\
&+
\frac{\Delta^2u}{u}
+2a\mathrm{Ric}\left(\frac{\nabla u}{u},\frac{\nabla u}{u}\right)
+\left(\frac{2a}{n}-1\right)
\frac{(\Delta u)^2}{u^2}
\nonumber\\
&-
\frac{4a(1+a)}{n}
\frac{\Delta u|\nabla u|^2}{u^3}
+
\frac{2a[(1+a)^2-na^2]}{n}
\frac{|\nabla u|^4}{u^4}.
\end{align*}
Since $\Delta^2u\geqslant0$ and $\mathrm{Ric}\geqslant0$, it follows that
\begin{align*}
u^{2a-2}
\operatorname{div}
\left(u^{2-2a}\nabla Z_a\right)
\geqslant{}&
\frac{2a-n}{n}Z_a^2
-\frac{2a(2+4a-n)}{n}
Z_a\frac{|\nabla u|^2}{u^2}
\nonumber\\
&+
\frac{a(2a+1)(4a+2-na)}{n}
\frac{|\nabla u|^4}{u^4}.
\tag{2.1}
\label{Z_a}
\end{align*}

\subsection{Key Lemma}

\begin{lemma}
Suppose that there exist $\delta\in\mathbb R_+$, depending only on $n$, and a sequence $\{R_m\}_{m=1}^\infty\subset\mathbb R_+$ with
$R_m\to+\infty$ such that, for every $m$, at some maximum point $x_m$ of the compactly supported function $\phi_m:=X\eta_m^4$,
\begin{equation*}
\operatorname{div}
\left(u^{2-2a}\nabla X\right)
\geqslant
\delta u^{2-2a}
\left(
X^2+\frac{|\nabla u|^4}{u^4}
\right).
\tag{2.2}
\label{X 1}
\end{equation*}
Then $X\leqslant0$.
\end{lemma}

\begin{proof}
Fix $m\in\mathbb N^*$, for simplicity, denote $\eta_m$ by $\eta$, $\phi _m$ by $\phi$, and $x_m$ by $x_0$. Then at $x_0$,
\[
    \nabla \phi=0,
    \qquad
    \Delta \phi\leqslant0.
\]
Hence
\begin{equation}
    \eta^2\nabla X+4\eta X\nabla\eta=0.
    \tag{2.3}
    \label{X 2}
\end{equation}
Furthermore,
\begin{align*}
0
\geqslant{}&
u^{2a-2}
\operatorname{div}
\left[
u^{2-2a}\nabla(X\eta^4)
\right]
\nonumber\\
={}&
u^{2a-2}
\operatorname{div}
\left(u^{2-2a}\nabla X\right)\eta^4
+8\eta^3\langle\nabla\eta,\nabla X\rangle
\nonumber\\
&+
\frac{8(1-a)}{u}
X\eta^3\langle\nabla u,\nabla\eta\rangle
+X\Delta\eta^4.
\end{align*}
Using \eqref{X 1} and \eqref{X 2}, we obtain
\begin{align*}
0\geqslant{}&
\delta\eta^4
\left(
X^2+\frac{|\nabla u|^4}{u^4}
\right)
-20X\eta^2|\nabla\eta|^2
\nonumber\\
&+
\frac{8(1-a)}{u}
X\eta^3
\langle\nabla u,\nabla\eta\rangle
+4X\eta^3\Delta\eta.
\end{align*}
Suppose first that $X(x_0)>0$. By Young's inequality,
\begin{align*}
\delta\eta^4
\left(
X^2+\frac{|\nabla u|^4}{u^4}
\right)
\leqslant
C X
\left(
\eta^2|\nabla\eta|^2
+\eta^3|\nabla\eta|\frac{|\nabla u|}{u}
+\eta^3(-\Delta\eta)
\right).
\end{align*}
Therefore, by \eqref{eta},
\begin{align*}
\delta\eta^4
\left(
X^2+\frac{|\nabla u|^4}{u^4}
\right)
\leqslant&
C X
\left(
R_m^{-2}\eta^2
+R_m^{-1}\eta^3\frac{|\nabla u|}{u}
\right)
\nonumber\\
\leqslant&
\frac{\delta}{2}X^2\eta^4
+CR_m^{-4}
+\frac{\delta}{2}
\frac{|\nabla u|^4}{u^4}\eta^4.
\end{align*}
Consequently, at $x_0$,
\[
    \eta^4
    \left(
    X^2+\frac{|\nabla u|^4}{u^4}
    \right)
    \leqslant CR_m^{-4}.
\]
Thus,
\[
    \phi(x)
    \leqslant \phi(x_0)
    \leqslant X(x_0)\eta(x_0)^2
    \leqslant CR_m^{-2}.
\]
If $X(x_0)\leqslant0$, the same estimate is immediate. Letting
$m\to+\infty$ yields $X(x)\leqslant0$.
\end{proof}

\subsection{A Continuity Argument}

\begin{proposition}
Let $(M^n,g)$ be a complete, connected, non-compact Riemannian manifold,
$n\geqslant5$, satisfying $\mathrm{Ric}\geqslant0$. If $u\in \mathcal{C}^4(M,\mathbb R_+)$, satisfying $\Delta^2u\geqslant0$ and \eqref{Delta u}, then
\[
    h:=\Delta u+\frac{2}{n-4}\frac{|\nabla u|^2}{u}
    \leqslant0.
\]
\end{proposition}

\begin{proof}
Define
\[
    T:=\{0\leqslant a\leqslant \frac{2}{n-4}: Z_a\leqslant0\}.
\]
Clearly $T$ is non-empty and closed, since $Z_a$ is continuous with respect to $a$. It remains to prove that $T$ is open. Suppose that $\frac{2}{n-4}\notin T$. Take $a_0\in T$ and
\[
    \varepsilon\in
    \left(
    0,\frac{2}{n-4}-a_0
    \right),
    \qquad
    a=a_0+\varepsilon.
\]
Since $Z_{a_0}\leqslant0$,
\[
    Z_a
    \leqslant
    \varepsilon\frac{|\nabla u|^2}{u^2}.
\]
Assume that there exists $q\in M$ satisfying $Z_a(q)>0$. Then for all large $R$, $Z_a(q)\eta(q)^4>0$. Let $x_0$ be a maximum point of $Z_a\eta^4$. Then $Z_a(x_0)>0$. From \eqref{Z_a} and Young's inequality,
\begin{align*}
u^{2a-2}
\operatorname{div}
\left(u^{2-2a}\nabla Z_a\right)
\geqslant{}&
-\left[
1+\frac{18na}{2-(n-4)a}
\right]Z_a^2
\nonumber
+
\frac{a[2-(n-4)a]}{2n}
\frac{|\nabla u|^4}{u^4}
\nonumber\\
\geqslant{}&
-\frac{36n}{2-(n-4)a}Z_a^2
+
\frac{a[2-(n-4)a]}{2n}
\frac{|\nabla u|^4}{u^4}.
\end{align*}
Set
\[
    C_1:=\frac{36n}{2-(n-4)a},
    \qquad
    C_2:=\frac{a[2-(n-4)a]}{2n},
    \qquad
    \delta:=\frac{C_2}{2}\in\mathbb{R}_+.
\]
Then at $x_0$,
\begin{align*}
u^{2a-2}
\operatorname{div}
\left(u^{2-2a}\nabla Z_a\right)
\geqslant{}&
\delta
\left(
Z_a^2+\frac{|\nabla u|^4}{u^4}
\right)
\nonumber
+
\left[
\delta-(C_1+\delta)\varepsilon^2
\right]
\frac{|\nabla u|^4}{u^4}\\
\geqslant{}&
\delta
\left(
Z_a^2+\frac{|\nabla u|^4}{u^4}
\right)
\nonumber
+
\varepsilon
\left[
\frac{9}{C_1}
-\left(
C_1+\frac{18}{C_1}
\right)\varepsilon
\right]
\frac{|\nabla u|^4}{u^4}.
\end{align*}
Since
\[
    \lim_{\varepsilon\to0^+}C_1
    =
    \frac{36n}{2-(n-4)a_0}\in\mathbb{R}_+,
\]
we may choose $\varepsilon\in\mathbb{R}_+$ sufficiently small, depending on $n$ and $a_0$, so that
\[
u^{2a-2}
\operatorname{div}
\left(u^{2-2a}\nabla Z_a\right)
\geqslant
\delta
\left(
Z_a^2+\frac{|\nabla u|^4}{u^4}
\right)
\]
at $x_0$ for all sufficiently large $R$. Then lemma 2.1 gives
$a\in T$, a contradiction. Hence $T$ is open, and therefore
$\frac2{n-4}\in T$.
\end{proof}

\begin{proposition}
Let $(M^n,g)$ be a complete, connected, non-compact Riemannian manifold with non-negative Ricci curvature and $n\geqslant5$. Let $u$ be a positive $\mathcal{C}^4$ function satisfying $\Delta^2u\geqslant u^{\frac{n+4}{n-4}}$ and \eqref{Delta u}. Then
\begin{equation*}
    \tilde h
    :=
    \Delta u
    +\frac{2}{n-4}\frac{|\nabla u|^2}{u}
    +\sqrt{\frac{n(n-4)}{n^2-4}}\,
      u^{\frac{n}{n-4}}
    \leqslant0.
\end{equation*}
\end{proposition}

\begin{proof}
We use the same continuity argument, now applied to $X_a$.
A direct computation yields
\begin{align*}
u^{2a-2}
\operatorname{div}
\left(u^{2-2a}\nabla X_a\right)
\geqslant{}&
\frac{2a-n}{n}X_a^2
+
\frac{a(2a+1)[2-(n-4)a]}{n}
\frac{|\nabla u|^4}{u^4}
\nonumber\\
&-
\frac{2a(2+4a-n)}{n}
X_a\frac{|\nabla u|^2}{u^2}
\nonumber\\
&+
\frac{2}{n}\left(
\frac{2}{n-4}-a\right)
\left(\frac{n^2}{n-4}-4a
\right)b_na
u^{-\frac{2n-12}{n-4}}|\nabla u|^2
\nonumber\\
&+
\left(
-\frac{4a}{n}
+\frac{2n-4}{n-4}
\right)b_na
u^{\frac{4}{n-4}}X_a
\nonumber\\
&+
\left[
1+
\frac{2(n-4)a-n^2}{n(n-4)}
b_n^2a^2
\right]
u^{\frac{8}{n-4}}.
\end{align*}
At a maximum point $x_0$ of $X_a\eta^4$ for large $R$, the same argument gives
\begin{align*}
u^{2a-2}
\operatorname{div}
\left(u^{2-2a}\nabla X_a\right)
\geqslant{}&
\delta
\left(
X_a^2+\frac{|\nabla u|^4}{u^4}
\right)
\nonumber
+
\varepsilon
\left[
\frac{9}{C_1}
-2\left(
C_1+\frac{18}{C_1}
\right)\varepsilon
\right]
\frac{|\nabla u|^4}{u^4}
\nonumber\\
&+
\left[
1+
\frac{2(n-4)a-n^2}{n(n-4)}
b_n^2a^2
-2\left(
C_1+\frac{18}{C_1}
\right)\varepsilon^2b_n^2
\right]
u^{\frac{8}{n-4}}.
\end{align*}
As $\varepsilon\to0^+$,
\[
1+
\frac{2(n-4)a-n^2}{n(n-4)}
b_n^2a^2
-2\left(
C_1+\frac{18}{C_1}
\right)\varepsilon^2b_n^2
\nonumber\longrightarrow
1+
\frac{2(n-4)a_0-n^2}{n(n-4)}
b_n^2a_0^2.
\]
The right-hand side is strictly decreasing as a function of $a_0$ and equal to $0$ at $\frac2{n-4}$, so it is positive. Hence $\varepsilon$ can be chosen sufficiently small, depending on $n$ and $a_0$, so that at $x_0$,
\[
u^{2a-2}
\operatorname{div}
\left(u^{2-2a}\nabla X_a\right)
\geqslant
\delta
\left(
X_a^2+\frac{|\nabla u|^4}{u^4}
\right).
\]
Then lemma 2.1 gives $X_a\leqslant0$.
\end{proof}
\section{Invariant Tensor Method}

\subsection{Construction of the Invariant Tensors}

Let $b$ be a constant to be determined. We introduce the trace-free tensor
\begin{equation*}
E_{ij}
=
u_{ij}
+
b\frac{u_i u_j}{u}
-
\frac1n
\left(
\Delta u+b\frac{|\nabla u|^2}{u}
\right)g_{ij}.
\end{equation*}
Then
\[
    E_{ij}=E_{ji},
    \qquad
    g^{ij}E_{ij}=0.
\]
Define
\begin{equation*}
    E_i:=\frac{E_{ij}u^j}{u},
    \qquad
    E:=\frac{E_i u^i}{u}.
\end{equation*}
Furthermore, set
\begin{equation*}
F_i
=
(\Delta u)_{,i}
+
\frac{n+2}{n}b\frac{\Delta u}{u}u_i
-
b\left(
1+\frac{n-2}{n}b
\right)
\frac{|\nabla u|^2}{u^2}u_i,
\end{equation*}
and
\[
    F:=\frac{F_i u^i}{u}.
\]
A direct calculation yields
\begin{align*}
E_{ij,}^{\ \; \;i}
=
\frac{n-2}{n}bE_j
+\frac{n-1}{n}F_j
+R_{ij}u^i,
\end{align*}
and
\begin{align*}
F_{i,}^{\ i}
=
-2b\left(1+\frac{n-2}{n}b\right)E
+\frac{n+2}{n}bF+G,
\end{align*}
where
\[
G
={}
\Delta^2u
+\frac{n+2}{n}b\frac{(\Delta u)^2}{u}
-\frac{2(n+2)b(1+b)}{n}
\frac{\Delta u|\nabla u|^2}{u^2}
\nonumber
+
b(3b+2)
\left(
1+\frac{n-2}{n}b
\right)
\frac{|\nabla u|^4}{u^3}.
\]
Differentiating $G$ gives
\begin{align*}
G_{,i}
={}&
4b
\left\{
-\frac{n+2}{n}(1+b)\frac{h}{u}
+
\left(
b+\frac{n-2}{n-4}
\right)
\left[
\frac{3(n-2)}{n}b+\frac{2(n-1)}{n}
\right]
\frac{|\nabla u|^2}{u^2}
\right\}E_i
\nonumber\\
&+
\frac{2(n+2)}{n}b
\left[
\frac{h}{u}
-
\left(
b+\frac{n-2}{n-4}
\right)
\frac{|\nabla u|^2}{u^2}
\right]F_i
+
\frac{f'(u)}{f(u)}Gu_i
\nonumber\\
&-
\frac{n+2}{n}b
\left[
\frac{n+4}{n}(2b+1)
+
\frac{uf'(u)}{f(u)}
\right]
\frac{h^2}{u^2}u_i
\nonumber\\
&+
b\left(
b+\frac{n-2}{n-4}
\right)
\left\{
\left[
\frac{2(n-1)(n+4)}{n^2}(2b+1)
+
\frac{n+2}{n}\frac{uf'(u)}{f(u)}
\right]
\frac{2h|\nabla u|^2}{u^3}u_i
\right.
\nonumber\\
&\qquad-
\left[
\left(
\frac{(n-2)(7n-4)}{n^2}b
+
\frac{2(3n^3-14n^2+22n-8)}
{n(n-4)^2}
\right)(2b+1)
\right.
\nonumber\\
&\left.\left.\qquad
+
\left(
\frac{3(n-2)}{n}b
+
\frac{2(n^2-4n+6)}{n(n-4)}
\right)
\frac{uf'(u)}{f(u)}
\right]
\frac{|\nabla u|^4}{u^4}u_i
\right\}.
\end{align*}

\subsection{A Differential Inequality}

We now choose
\[
    b=-\frac{n-2}{n-4},
\]
and write
\begin{equation*}
    R_0:=\frac{R_{ij}u^iu^j}{u}\geqslant0.
\end{equation*}
Then
\begin{equation*}
E_{ij}
=
u_{ij}
-\frac{n-2}{n-4}\frac{u_i u_j}{u}
-\frac1n
\left(
\Delta u-\frac{n-2}{n-4}\frac{|\nabla u|^2}{u}
\right)g_{ij}.
\end{equation*}
The identities above reduce to
\begin{equation*}
E_{ij,}^{\ \; \;i}
=
-\frac{(n-2)^2}{n(n-4)}E_j
+\frac{n-1}{n}F_j
+R_{ij}u^i,
\tag{3.1}
\label{E_{ij,}^{i}}
\end{equation*}
and
\begin{equation*}
F_i
=
h_i-\frac{4}{n-4}E_i
-\frac{n}{n-4}\frac{h}{u}u_i.  
\end{equation*}
Furthermore,
\begin{equation*}
F_{i,}^{\ i}
=
-\frac{8(n-2)}{n(n-4)^2}E
-\frac{(n-2)(n+2)}{n(n-4)}F+G,    
\end{equation*}
where
\begin{equation*}
G
=
\Delta^2u
-
\frac{(n-2)(n+2)}{n(n-4)}
\frac{h^2}{u}.
\end{equation*}
Moreover,
\begin{equation*}
G_{,i}
={}
\frac{(n-2)(n+2)}{n(n-4)}
\left[
-\frac{8}{n-4}\frac{h}{u}E_i
-\frac{2h}{u}F_i
+
\left(
\frac{uf'(u)}{f(u)}
-\frac{n+4}{n-4}
\right)
\frac{h^2}{u^2}u_i
\right]
+
\frac{f'(u)}{f(u)}Gu_i.
\end{equation*}
For the critical nonlinearity, the invariant tensors possess the required differential invariance. Furthermore, \eqref{eq} and Proposition 2.3 give
\begin{equation*}
    G
    =
    \frac{(n-2)(n+2)}{n(n-4)}
    \frac{\tilde h}{u}
    \left(
    2\sqrt{ \frac{n(n-4)}{n^2-4}}u^{\frac{n}{n-4}}
    -\tilde h
    \right)
    \leqslant0.
    \tag{3.2}
    \label{G 1}
\end{equation*}

A further computation shows that when $\Delta^2u\geqslant u^{\frac{n+4}{n-4}}$, the non-positive smooth function $\tilde h$ satisfies the following second-order differential inequality
\begin{align*}
&\Delta \tilde h
- \frac{4}{n-4}\frac{\langle \nabla u,\nabla \tilde h\rangle}{u}
- \frac{(n-2)(n+2)}{n(n-4)}
\left(
\frac{n^2}{n^2-4}\Delta u
+2\sqrt{ \frac{n(n-4)}{n^2-4}}u^{\frac{n}{n-4}}
-\tilde h
\right)
\frac{\tilde h}{u}
\nonumber\\
&\geqslant{}
\Delta \tilde h
- \frac{4}{n-4}\frac{\langle \nabla u,\nabla \tilde h\rangle}{u}
- \frac{n}{n-4}\frac{\Delta u}{u}\tilde h
- G
=
\frac{4}{n-4}\left(
\frac{1}{u}E_{ij}E^{ij}+R_0
\right)
\geqslant{}0.    
\end{align*}
By Proposition 2.3 and the strong maximum principle, the set $\{\tilde h=0\}$ is open. On the other hand, since $\tilde h$ is continuous, the set $\{\tilde h=0\}$ is closed. Therefore, by the connectedness of the manifold, either $\tilde h<0$ everywhere or $\tilde h\equiv0$. In the case $\tilde h\equiv 0$, from the above second-order equation, we have $E_{ij}\equiv0$. Consequently,
\[
    F_i
    =
    \tilde h_{,i}
    -\frac{4}{n-4}E_i
    -\frac{n}{n-4}\frac{\tilde h}{u}u_i
    =0.
\]
From \eqref{E_{ij,}^{i}}, $R_{ij}u^i=0$. Set $v=u^{-\frac{2}{n-4}}$. Then
\begin{equation*}
    \nabla^2 v=\frac{\Delta v}{n}g,
    \qquad
    \mathrm{Ric}(dv)=0.
\end{equation*}
Taking the divergence and using the Bochner formula gives
\[
    d\Delta v=0.
\]
Thus $\Delta v\equiv C$ for some constant $C$, and therefore
\begin{equation*}
    \nabla^2 v=\frac{C}{n}g.
\end{equation*}
On the other hand, Proposition 2.3 implies
\[
    \Delta u
    <
    -\frac{2}{n-4}\frac{|\nabla u|^2}{u},
\]
and hence
\begin{align*}
C
=\Delta v
={}&
-\frac{2v}{n-4}
\left(
\frac{\Delta u}{u}
-\frac{n-2}{n-4}
\frac{|\nabla u|^2}{u^2}
\right)
>
\frac{2n}{(n-4)^2}
\frac{|\nabla u|^2}{u^2}v
\geqslant0.
\end{align*}
By Tashiro's theorem, $(M,g)$ is isometric to Euclidean space. So it remains to solve
\[
    \partial_i\partial_jv=\frac{C}{n}\delta_{ij}.
\]
For $i\neq j$, $\partial_i\partial_jv=0$, so $\partial_jv$ depends only on $x_j$. Moreover, $\partial_j^3v=0$. Hence
\[
    \partial_jv=\frac{C}{n}x_j+b_j,
\]
and consequently,
\begin{equation*}
    v=c|x|^2+\langle b,x\rangle+a,
    \qquad
    a\in\mathbb{R},\quad b\in\mathbb{R}^n,\quad c\in\mathbb{R}_+.
\end{equation*}
Completing the square,
\[
    v=c|x-x_0|^2+d.
\]
Since $v>0$, we have $d\in\mathbb{R}_+$. Thus
\begin{equation*}
    v=A\left(1+\lambda^2|x-x_0|^2\right),
    \qquad
    A,\lambda\in\mathbb{R}_+,\quad x_0\in\mathbb{R}^n.
\end{equation*}
Substitution into the critical equation gives
\[
    u(x)
    =
    \frac{
    \left[
    \lambda^2n(n-4)(n^2-4)
    \right]^{\frac{n-4}{8}}
    }
    {\left(
    1+\lambda|x-x_0|^2
    \right)^{\frac{n-4}{2}}}.
\]
Then we obtain the following corollary.
\begin{corollary}
Let $(M^n,g)$ be a complete, connected, non-compact Riemannian manifold with non-negative Ricci curvature and $n\geqslant5$. Let $u$ be a positive $\mathcal{C}^4$ function satisfying $\Delta^2u\geqslant u^{\frac{n+4}{n-4}}$, \eqref{Delta u} and that at some point $q\in M$,
\[
    \Delta u(q)
    \geqslant
    -\frac{2}{n-4}\frac{|\nabla u(q)|^2}{u(q)}
    -\sqrt{\frac{n(n-4)}{n^2-4}}\,
      {u(q)}^{\frac{n}{n-4}}.
\]
Then $(M^n,g)$ is isometric to
$(\mathbb{R}^n,g_E)$. Moreover, $u$ is of the form
\begin{equation*}
    u(x)
    =
    \frac{
    \left[\lambda^2 n(n-4)(n^2-4)\right]^{\frac{n-4}{8}}
    }
    {\left(1+\lambda|x-x_0|^2\right)^{\frac{n-4}{2}}},
    \qquad
    \lambda\in\mathbb{R}_+,\quad x_0\in\mathbb{R}^n.
\end{equation*}    
\end{corollary}
To get required differential inequality, we first prove the following algebraic inequality.
\begin{lemma}
For the trace-free tensor $E_{ij}$ defined above,
\begin{equation*}
    \frac{|\nabla u|^2}{u^2}E_{ij}E^{ij}
    \geqslant
    \frac{n}{n-1}E_iE^i.
\end{equation*}
\end{lemma}

\begin{proof}
At any fixed point, choose the normal coordinate system. Since $\sum_{i=1}^nE_{ii}=0$, it suffices to prove that for every symmetric matrix $A\in\operatorname{Sym}_n(\mathbb{R})$ with $\operatorname{tr} A=0$ and every
$x\in\mathbb{R}^n$,
\begin{equation*}
    |x|^2\|A\|_F^2
    \geqslant
    \frac{n}{n-1}|Ax|^2.
\end{equation*}
The case $x=0$ is trivial. If $x\neq0$, by homogeneity we may assume
$|x|=1$, and after an orthogonal change of coordinates we may assume
$x=(1,0,\ldots,0)^T$, and
\[
    A_{ij} = 0,
    \qquad
    2 \leqslant i \not = j \leqslant n.
\]
Thus
\begin{align*}
|x|^2\|A\|_F^2-\frac{n}{n-1}|Ax|^2
={}&
\sum_{i=1}^n {A_{ii}}^2 + 2\sum_{i=2}^n {A_{1i}}^2 -  \frac{n}{n-1}\sum_{i=1}^n {A_{1i}}^2 
\\
={}&
\sum_{i=2}^nA_{ii}^2
+\frac{n-2}{n-1}\sum_{i=2}^nA_{1i}^2
-\frac1{n-1}A_{11}^2
\\
\geqslant{}&
\sum_{i=2}^nA_{ii}^2
-\frac1{n-1}
\left(\sum_{i=2}^nA_{ii}\right)^2
\\
={}&
\frac1{n-1}
\sum_{2\leqslant i<j\leqslant n}
(A_{ii}-A_{jj})^2
\geqslant0.
\end{align*}
\end{proof}
Now we can derive the desired differential inequality. Let $\beta$ be a constant to be chosen. Direct calculations yield
\begin{align*}
u^\beta
\left(
u^{-\beta}\tilde h E_i
\right)_{,}^{\ i}
={}&
\tilde h
\left(
\frac1u E_{ij}E^{ij}+R_0
\right)
+\frac{4}{n-4}E_iE^i
+E_iF^i
\nonumber\\
&+
\left[
\frac{2(3n-2)}{n(n-4)}-\beta
\right]\tilde h E
+
\frac{n-1}{n}\tilde h F,
\tag{3.3}
\label{3.3}
\end{align*}
\begin{align*}
u^\beta
\left(
u^{-\beta-1}|\nabla u|^2E_i
\right)_{,}^{\ i}
={}&
\frac{|\nabla u|^2}{u}
\left(
\frac1u E_{ij}E^{ij}+R_0
\right)
+2E_iE^i
+\frac{2}{n}h E
\nonumber\\
&+
\left[
\frac{4(n-1)}{n(n-4)}-\beta
\right]
\frac{|\nabla u|^2}{u}E
+
\frac{n-1}{n}\frac{|\nabla u|^2}{u}F,
\tag{3.4}
\label{3.4}
\end{align*}
\begin{align*}
u^\beta
\left(
u^{-\beta+\frac{n}{n-4}}E_i
\right)_{,}^{\ i}
={}&
u^{\frac{n}{n-4}}
\left\{
\frac1uE_{ij}E^{ij}+R_0
+
\left[
\frac{2(3n-2)}{n(n-4)}-\beta
\right]E
+\frac{n-1}{n}F
\right\},
\tag{3.5}
\label{3.5}
\end{align*}
\begin{align*}
u^\beta
\left(
u^{-\beta}\tilde h F_i
\right)_{,}^{\ i}
={}&
\frac{4}{n-4}E_iF^i
+F_iF^i
-\frac{8(n-2)}{n(n-4)^2}\tilde h E
+
\left[
\frac{4}{n(n-4)}-\beta
\right]\tilde h F
+\tilde h G,
\tag{3.6}
\label{3.6}
\end{align*}
\begin{align*}
u^\beta
\left(
u^{-\beta-1}|\nabla u|^2F_i
\right)_{,}^{\ i}
={}&
2E_iF^i
-\frac{8(n-2)}{n(n-4)^2}
\frac{|\nabla u|^2}{u}E
+\frac{2}{n}hF
\nonumber\\
&+
\left[
-\frac{2(n-2)}{n(n-4)}-\beta
\right]
\frac{|\nabla u|^2}{u}F
+
\frac{|\nabla u|^2}{u}G,
\tag{3.7}
\label{3.7}
\end{align*}
\begin{align*}
u^\beta
\left(
u^{-\beta+\frac{n}{n-4}}F_i
\right)_{,}^{\ i}
=
u^{\frac{n}{n-4}}
\left\{
-\frac{8(n-2)}{n(n-4)^2}E
+
\left[
\frac{4}{n(n-4)}-\beta
\right]F
+G
\right\},
\tag{3.8}
\label{3.8}
\end{align*}
and
\begin{align*}
u^\beta(u^{-\beta}Gu_i)_{,}^{\ i}
={}&
\frac{(n-2)(n+2)}{n(n-4)}
\frac{h}{u}
\left[
-\frac8{n-4}uE
-2uF
+
\left(
\frac{uf'(u)}{f(u)}
-\frac{n+4}{n-4}
\right)
\frac{h|\nabla u|^2}{u}
\right]
\nonumber\\
&+
\left\{
h+
\left[
\frac{uf'(u)}{f(u)}
-\left(
\frac2{n-4}+\beta
\right)
\right]
\frac{|\nabla u|^2}{u}
\right\}G.
\tag{3.9}
\label{3.9}
\end{align*}
We set
\begin{equation*}
    c_n
    :=
    \sqrt{\frac{n(n-4)}{n^2-4}}.
\end{equation*}
Taking a suitable linear combination of \eqref{3.3}--\eqref{3.9},
and using Proposition 2.3 together with Lemma 3.2, we obtain a constant
$\delta\in\mathbb{R}_+$, depending only on $n$, such that
\begin{align*}
&u^\beta
\left\{
u^{-\beta}
\left[
\left(
c_1\tilde h
+c_2\frac{|\nabla u|^2}{u}
+c_3u^{\frac{n}{n-4}}
\right)E_i
+
\left(
c_4\tilde h
+c_5\frac{|\nabla u|^2}{u}
+c_6u^{\frac{n}{n-4}}
\right)F_i
+c_7Gu_i
\right]
\right\}_{,}^{\ i}
\nonumber\\
&\geqslant
\frac{2(n^2-3)c_n}
{(n-1)(n-4)}
u^{\frac{n}{n-4}}
\left(
\frac1uE_{ij}E^{ij}+R_0
\right)
\nonumber\\
&\quad
+
\frac{4(n^2-5n+8)}
{(n-1)(n-4)^2}E_iE^i
+
\frac{4}{(n-1)(n-4)}E_iF^i
+F_iF^i
-\frac1{n-4}\frac{|\nabla u|^2}{u}G
\nonumber\\
&\geqslant
\delta
\left(
u^{\frac4{n-4}}E_{ij}E^{ij}
+E_iE^i
+F_iF^i
-\frac{|\nabla u|^2}{u}G
\right),
\tag{3.10}
\label{3.10}
\end{align*}
where
\begin{align*}
c_1&=-\frac{2(n^2-3)}{(n-1)(n-4)},
&
c_2&=\frac{4(n-1)}{(n-4)^2},
&
c_3&=\frac{2(n^2-3)c_n}{(n-1)(n-4)},
&
c_4&=1,
\nonumber\\
c_5&=\frac{n-1}{n-4},
&
c_6&=-c_n,
&
c_7&=-1,
&
\beta&=\frac{2}{n-4}.
\end{align*}
\section{Integral Estimates and Proofs of the Main Theorems}

\subsection{Integral Estimate}

Let $\gamma>6$. Multiplying \eqref{3.10} by
$u^{-\frac{2}{n-4}}\eta^\gamma$, integrating by parts and using \eqref{eta}, we obtain
\begin{align*}
\Delta
:={}&
\int_M
\left[
u^{\frac{2}{n-4}}E_{ij}E^{ij}
+
u^{-\frac{2}{n-4}}
(E_iE^i+F_iF^i)
-
u^{- \frac{2}{n-4}-1}
|\nabla u|^2G
\right]\eta^\gamma
\nonumber\\
\leqslant{}&
C\int_M
u^{-\frac{2}{n-4}}
\left[
\left(
-\tilde h
+\frac{|\nabla u|^2}{u}
+u^{\frac{n}{n-4}}
\right)
\left(
|E_i\eta^i|+|F_i\eta^i|
\right)
-G|u_i\eta^i|
\right]\eta^{\gamma-1}
\nonumber\\
\leqslant{}&
\frac12\Delta
+
CR^{-2}
\int_M
u^{-\frac{2}{n-4}}
\left(
\tilde h^2
+\frac{|\nabla u|^4}{u^2}
+u^{\frac{2n}{n-4}}
\right)
\eta^{\gamma-2}
\nonumber\\
&+
CR^{-2}
\int_M
u^{-\frac{2}{n-4}+1}
(-G)\eta^{\gamma-2}.
\end{align*}
Using \eqref{G 1}, Proposition 2.2 and Proposition 2.3,
\begin{align*}
\Delta
\leqslant{}&
CR^{-2}
\int_M
u^{-\frac{2}{n-4}}
\left(
\tilde h^2
+\frac{|\nabla u|^4}{u^2}
+u^{\frac{2n}{n-4}}
\right)
\eta^{\gamma-2}
\nonumber\\
\leqslant{}&
CR^{-2}
\int_M
u^{-\frac{2}{n-4}}
\left(
h^2+\frac{|\nabla u|^4}{u^2}
\right)
\eta^{\gamma-2}
\leqslant{}
C_0R^{-2}
\int_M
u^{-\frac{2}{n-4}}
(\Delta u)^2
\eta^{\gamma-2}.
\tag{4.1}
\label{Delta}
\end{align*}
Moreover, the Bishop--Gromov volume comparison theorem gives
\begin{equation*}
    \operatorname{Vol}(B_R)\leqslant\omega_nR^n,
\tag{4.2}
\label{vol}
\end{equation*}
since $(M^n,g)$ is complete and $\mathrm{Ric}\geqslant0$.

\subsection{Proof of Theorem 1.1: The Case \texorpdfstring{$n=5$}{n}}

Consider the differential identity
\begin{align*}
u^{\frac{2}{n-4}}
\left(
u^{-\frac{2}{n-4}}\Delta u\,u_i
\right)_,^{\ i}
={}&
(\Delta u)^2+uF
+
\frac{n^2-2n-4}{n(n-4)}
\frac{\Delta u|\nabla u|^2}{u}
+
\frac{4(n-2)}{n(n-4)^2}
\frac{|\nabla u|^4}{u^2}.
\end{align*}
For $n=5$ this becomes
\begin{equation*}
u^2
\left(
u^{-2}\Delta u\,u_i
\right)_,^{\ i}
=
(\Delta u)^2+uF
+\frac{11}{5}
\frac{\Delta u|\nabla u|^2}{u}
+\frac{12}{5}
\frac{|\nabla u|^4}{u^2}.
\end{equation*}
Hence, for some $\delta_0\in\mathbb{R}_+$ ,
\begin{equation*}
u^2
\left(
u^{-2}\Delta u\,u_i
\right)_,^{\ i}
\geqslant
\delta_0
\left[
(\Delta u)^2
+\frac{|\nabla u|^4}{u^2}
\right]
-|F_i u^i|.
\end{equation*}
Integrating by parts and using \eqref{eta},
\begin{align*}
&\int_M
u^{-2}
\left[
(\Delta u)^2+\frac{|\nabla u|^4}{u^2}
\right]\eta^{\gamma-2}
\nonumber\\
&\leqslant{}
C\int_M
u^{-2}(-\Delta u)|\nabla u||\nabla\eta|
\eta^{\gamma-3}
+
C\int_Mu^{-2}|F_iu^i|\eta^{\gamma-2}
\nonumber\\
&\leqslant{}
\frac12
\int_Mu^{-2}(\Delta u)^2\eta^{\gamma-2}
+
CR^{-2}
\int_Mu^{-2}|\nabla u|^2\eta^{\gamma-4}
+
C\int_Mu^{-2}|F_iu^i|\eta^{\gamma-2}
\nonumber\\
&\leqslant{}
\frac12 
\int_M
u^{-2}
\left[
(\Delta u)^2+\frac{|\nabla u|^4}{u^2}
\right]\eta^{\gamma-2}
+
CR^{-4}\int_M\eta^{\gamma-6}
+
C\int_Mu^{-2}|F_iu^i|\eta^{\gamma-2}.
\end{align*}
Then by Young's inequality,
\begin{align*}
R^{-2}
\int_M
u^{-2}
\left[
(\Delta u)^2+\frac{|\nabla u|^4}{u^2}
\right]\eta^{\gamma-2}
\leqslant{}&
CR^{-6}\int_M\eta^{\gamma-6}
+
CR^{-2}
\int_Mu^{-2}|F_iu^i|\eta^{\gamma-2}.
\end{align*}
Consequently,
\begin{align*}
R^{-2}
\int_M
u^{-2}
\left[
(\Delta u)^2+\frac{|\nabla u|^4}{u^2}
\right]\eta^{\gamma-2}
\leqslant{}&
\varepsilon\Delta
+
\frac12R^{-2}
\int_M
\frac{|\nabla u|^4}{u^4}\eta^{\gamma-2}
+
CR^{-6}
\int_M\eta^{\gamma-6}.
\end{align*}
Therefore,
\begin{align*}
R^{-2}
\int_M
u^{-2}
\left[
(\Delta u)^2+\frac{|\nabla u|^4}{u^2}
\right]\eta^{\gamma-2}
\leqslant
2\varepsilon\Delta
+
CR^{-6}
\int_M\eta^{\gamma-6}.
\end{align*}
We may choose $\varepsilon\in\mathbb{R}_+$ sufficiently small, depending on $C_0$, and combine \eqref{Delta}, so that
\begin{align*}
\Delta
\leqslant
C_0R^{-2}
\int_Mu^{-2}(\Delta u)^2\eta^{\gamma-2}
\leqslant
\frac12\Delta
+
CR^{-6}\int_M\eta^{\gamma-6}.
\end{align*}
Thus by \eqref{vol},
\begin{align*}
\int_{B_R}
\left(
u^2E_{ij}E^{ij}+u^{-2}F_iF^i
\right)
\leqslant
\Delta
\leqslant
CR^{-6}\int_M\eta^{\gamma-6}
\leqslant CR^{-1}.
\end{align*}
Letting $R\to+\infty$, we obtain $E_{ij}\equiv0$ and $F_i\equiv0$. The rigidity argument in the subsection 3.2 then yields the required classification.

\subsection{Proof of Theorem 1.2: The case \texorpdfstring{$n=6$}{n}}

Integration by parts and using \eqref{eq}, \eqref{eta} gives
\begin{align*}
R^{-2}
\int_M
\frac{(\Delta u)^2}{u}\eta^{\gamma-2}
\leqslant{}&
R^{-2}
\int_M
\frac{\Delta u|\nabla u|^2}{u^2}\eta^{\gamma-2}
+
CR^{-3}
\int_M
\frac{(-\Delta u)|\nabla u|}{u}\eta^{\gamma-3}
\nonumber\\
&-
R^{-2}
\int_M
(\Delta u)_{,i}(\ln u)_{,}^{\ i}\eta^{\gamma-2}
\nonumber\\
\leqslant{}&
R^{-2}
\int_M
u^5\ln u\,\eta^{\gamma-2}
+
CR^{-3}
\int_M
\frac{(-\Delta u)|\nabla u|}{u}\eta^{\gamma-3}
\nonumber\\
&+
(\gamma-2)R^{-2}
\int_M
\ln u\,\eta^{\gamma-3}\eta_i(\Delta u)_{,}^{\ i}.
\end{align*}
After another integration by parts,
\begin{align*}
R^{-2}
\int_M
\frac{(\Delta u)^2}{u}\eta^{\gamma-2}
\leqslant{}&
R^{-2}
\int_M
u^5\ln u\,\eta^{\gamma-2}
+
CR^{-3}
\int_M
\frac{(-\Delta u)|\nabla u|}{u}\eta^{\gamma-3}
\nonumber\\
&+
CR^{-4}
\int_M
|\Delta u\ln u|\eta^{\gamma-4}.
\end{align*}
Using Proposition 2.2 and Young's inequality,
\begin{align*}
R^{-2}
\int_M
\frac{(\Delta u)^2}{u}\eta^{\gamma-2}
\leqslant{}&
R^{-2}
\int_M
u^5\ln u\,\eta^{\gamma-2}
+
CR^{-3}
\int_M
u^{-\frac12}(-\Delta u)^{\frac32}\eta^{\gamma-3}
\nonumber\\
&+
CR^{-4}
\int_M
|\Delta u\ln u|\eta^{\gamma-4}
\nonumber\\
\leqslant{}&
R^{-2}
\int_Mu^5\ln u\,\eta^{\gamma-2}
+\frac12R^{-2}
\int_M\frac{(\Delta u)^2}{u}\eta^{\gamma-2}
\nonumber\\
&+
CR^{-6}\int_Mu\eta^{\gamma-6}
+
CR^{-6}\int_Mu(\ln u)^2\eta^{\gamma-6}.
\end{align*}
Thus we have
\begin{align*}
R^{-2}
\int_M
\frac{(\Delta u)^2}{u}\eta^{\gamma-2}
&\leqslant{}
2R^{-2}
\int_Mu^5\ln u\eta^{\gamma-2}
+
CR^{-6}
\int_Mu\eta^{\gamma-6}
+
CR^{-6}
\int_Mu(\ln u)^2\eta^{\gamma-6}
\nonumber\\
&\leqslant{}
2R^{-2}
\int_Mu^5\ln u\eta^{\gamma-2}
+
CR^{-6}
\int_M
\left(
u^{\frac12}
+
u^{\frac32}
\right)\eta^{\gamma-6}.
\end{align*}
Combining this estimate with \eqref{Delta}, we obtain
\begin{equation*}
    \Delta
    \leqslant{}
    2C_0R^{-2}
    \int_Mu^5\ln u\eta^{\gamma-2}
    +
    CR^{-6}
    \int_M
    \left(
    u^{\frac12}
    +
    u^{\frac32}
    \right)\eta^{\gamma-6}.
\tag{4.3}
\label{n=6}
\end{equation*}
Then by \eqref{vol} and Hölder's inequality,
\begin{align*}
\Delta
\leqslant{}&
CR^{-2}
\int_Mu^6
+
CR^{-\frac12}
\left(
\int_Mu^6
\right)^{\frac1{12}}
+
CR^{-\frac32}
\left(
\int_Mu^6
\right)^{\frac14}.
\end{align*}
Since $u\in \mathcal{L}^6(M)$, letting $R\to+\infty$ gives $E_{ij}\equiv0$ and $F_i\equiv0$.

\subsection{An Alternative Growth Condition for \texorpdfstring{$n=6$}{n}}

For every $n\geqslant5$, using \eqref{eq}, \eqref{eta}, $\Delta u\leqslant0$, and Young's inequality,
\begin{align*}
\int_M
u^{\frac{n+4}{n-4}}\eta^{n+4}
&=
\int_M\Delta u\,\Delta(\eta^{n+4})
\nonumber\\
&\leqslant
CR^{-2}
\int_M(-\Delta u)\eta^{n+2}
=
CR^{-2}
\int_M-u\Delta(\eta^{n+2})
\nonumber\\
&\leqslant
CR^{-4}
\int_Mu\eta^n
\nonumber\\
&\leqslant
\frac12
\int_M
u^{\frac{n+4}{n-4}}\eta^{n+4}
+
CR^{-\frac{n+4}{2}}
\int_M\eta^{\frac{(n+4)}2}.
\end{align*}
Thus by \eqref{vol},
\begin{equation*}
    \int_{B_R}
    u^{\frac{n+4}{n-4}}
    \leqslant
    CR^{\frac n2-2}.
\tag{4.4}
\label{n+4/n-4}
\end{equation*}
In particular, when $n=6$,
\begin{equation*}
    \int_{B_R}u^5\leqslant CR.
\tag{4.5}
\label{5}
\end{equation*}
Suppose there exists $p\in M$ such that for large $R$,
\begin{equation*}
    \sup_{B_R}\ln u=O(R),
    \tag{4.6}
    \label{growth 1}
\end{equation*}
Let $\eta$ be the standard cut-off function associated with $\frac R2$. Then using \eqref{vol}, \eqref{n=6}, \eqref{5}, \eqref{growth 1} and Hölder's inequality, for large $R$,
\begin{align*}
\Delta
\leqslant{}&
2C_0R^{-2}
\int_{B_R}
u^5\ln u\,\eta^{\gamma-2}
+
CR^{-6}
\int_{B_R}
\left(
u^{\frac12}+u^{\frac32}
\right)
\eta^{\gamma-6}
\nonumber\\
\leqslant{}&
2C_0R^{-2}
\left(
\sup_{B_R}\ln u
\right)
\int_{B_R}u^5
+
CR^{-\frac35}
\left(
\int_{B_R}u^5
\right)^{\frac1{10}}+
CR^{-\frac95}
\left(
\int_{B_R}u^5
\right)^{\frac3{10}}
\nonumber\\
\leqslant{}&
CR^{-1}
\int_{B_R}u^5
+
CR^{-\frac12}
+
CR^{-\frac32}
\leqslant{}
C.
\end{align*}
Consequently,
\begin{equation*}
    I_R
    :=
    \int_{B_{R/2}}
    u^{-1}
    (E_iE^i+F_iF^i)
    -
    u^{-2}
    |\nabla u|^2G
    \leqslant
    C.
\end{equation*}
Now we can use the Caccioppoli trick. Since $I_R$ is bounded and increasing with respect to $R$, the limit exists as $R\to+\infty$, which is denoted by $I$. Clearly, we have
\[
    I=
    \int_M
    u^{-1}
    (E_iE^i+F_iF^i)
    -
    u^{-2}
    |\nabla u|^2G
    \leqslant{}
    C. 
\]
By \eqref{eta} and \eqref{G 1}, for any $A\in\mathbb{R}_+$,
\begin{align*}
\Delta&
\leqslant{}
C\int_{B_R\setminus B_{R/2}}
u^{-1}
\left[
\left(
-\tilde h
+\frac{|\nabla u|^2}{u}
+u^3
\right)
\left(
|E_i\eta^i|+|F_i\eta^i|
\right)
-G|u_i\eta^i|
\right]\eta^{\gamma-1}
\nonumber\\
&\leqslant{}
A\int_{B_R\setminus B_{R/2}}
\left[
u^{-1}
(E_iE^i+F_iF^i)
-
u^{-2}
|\nabla u|^2G
\right]\eta^\gamma
\nonumber\\
&
\quad
+
CA^{-1}R^{-2}
\int_{B_R\setminus B_{R/2}}
u^{-1}
\left(
\tilde h^2
+\frac{|\nabla u|^4}{u^2}
+u^6
\right)
\eta^{\gamma-2}
\nonumber\\
&\leqslant{}
A\left(I-I_R\right)
+
CA^{-1}R^{-2}
\int_M
\frac{(\Delta u)^2}{u}
\eta^{\gamma-2}.
\end{align*}
We have proved
\begin{align*}
R^{-2}
\int_M
\frac{(\Delta u)^2}{u}\eta^{\gamma-2}
&\leqslant{}
2R^{-2}
\int_Mu^5\ln u\eta^{\gamma-2}
+
CR^{-6}
\int_M
\left(
u^{\frac12}
+
u^{\frac32}
\right)\eta^{\gamma-6}
\leqslant{}
C.
\end{align*}
Thus
\[
    \int_{B_{R/2}}
    uE_{ij}E^{ij}
    +
    u^{-1}F_iF^i
    \leqslant{}
    \Delta
    \leqslant{}
    A(I-I_R)
    +
    CA^{-1}.
\]
Then letting $R\to+\infty$ gives
\[
    \int_M
    uE_{ij}E^{ij}
    +
    u^{-1}F_iF^i
    \leqslant{}
    CA^{-1}.
\]
Letting $A\to+\infty$ gives $E_{ij}\equiv0$ and $F_i\equiv0$. Finally, the rigidity argument in the subsection 3.2 yields the required classification. Furthermore, by the Caccioppoli trick, it is easy to check that it suffices to assume that
\[
\int_{B_R} u^6 = O(R^2),
\]
in the subsection 4.3 for some $p\in M$ and large $R$.

\subsection{Proof of Theorem 1.2: The Case \texorpdfstring{$n\geqslant7$}{n}}

Consider the differential identity
\begin{align*}
u^{\frac{2}{n-4}}
\left(
u^{-\frac{2}{n-4}-1}|\nabla u|^2u_i
\right)_{,}^{\ i}
={}&
\frac{n-6}{n-4}\frac{|\nabla u|^4}{u^2}
+
\frac{n+2}{n}
\frac{h|\nabla u|^2}{u}
+2uE.
\end{align*}
For $n\geqslant7$, integration by parts yields
\begin{align*}
\int_M
u^{-\frac{2}{n-4}}
\frac{|\nabla u|^4}{u^2}
\eta^{\gamma-2}
\leqslant{}&
C
\int_M
u^{-\frac{2}{n-4}}
|E_iu^i|\eta^{\gamma-2}
+
C
\int_M
u^{-\frac{2}{n-4}}h^2\eta^{\gamma-2}
\nonumber\\
&+
C
\int_M
u^{-\frac{2}{n-4}-1}
|\nabla u|^3|\nabla\eta|
\eta^{\gamma-3}.
\end{align*}
By \eqref{eta} and Young's inequality,
\begin{align*}
\int_M
u^{-\frac{2}{n-4}}
\frac{|\nabla u|^4}{u^2}
\eta^{\gamma-2}
\leqslant{}&
C
\int_M
u^{-\frac{2}{n-4}}
|E_iu^i|\eta^{\gamma-2}
+
C
\int_M
u^{-\frac{2}{n-4}}h^2\eta^{\gamma-2}
\nonumber\\
&+
\frac12
\int_M
u^{-\frac{2}{n-4}}
\frac{|\nabla u|^4}{u^2}
\eta^{\gamma-2}
+
CR^{-2}
\int_M
u^{-\frac{2}{n-4}}
|\nabla u|^2\eta^{\gamma-4}.
\end{align*}
Therefore,
\begin{align*}
\int_M
u^{-\frac{2}{n-4}}
\frac{|\nabla u|^4}{u^2}
\eta^{\gamma-2}
\leqslant{}&
C
\int_M
u^{-\frac{2}{n-4}}
|E_iu^i|\eta^{\gamma-2}
+
C
\int_M
u^{-\frac{2}{n-4}}h^2\eta^{\gamma-2}
\nonumber\\
&+
CR^{-2}
\int_M
u^{-\frac{2}{n-4}}
|\nabla u|^2\eta^{\gamma-4}.
\tag{4.7}
\label{n>=7 1}
\end{align*}
Next, consider the differential identity
\begin{align*}
u^{\frac{2}{n-4}}
\left(
u^{-\frac{2}{n-4}}hu_i
\right)_{,}^{\ i}
=
h^2
+
\frac{h|\nabla u|^2}{u}
+
\frac{4}{n-4}uE
+
uF.
\end{align*}
Integration by parts gives
\begin{align*}
\int_M
u^{-\frac{2}{n-4}}h^2\eta^{\gamma-2}
\leqslant{}&
C
\int_M
u^{-\frac{2}{n-4}}
\left(
|E_iu^i|+|F_iu^i|
\right)\eta^{\gamma-2}
+
C
\int_M
u^{-\frac{2}{n-4}-1}
(-h)|\nabla u|^2\eta^{\gamma-2}
\nonumber\\
&+
C
\int_M
u^{-\frac{2}{n-4}}
(-h)|\nabla u||\nabla\eta|
\eta^{\gamma-3}.
\end{align*}
Again applying \eqref{eta} and Young's inequality,
\begin{align*}
\int_M
u^{-\frac{2}{n-4}}h^2\eta^{\gamma-2}
\leqslant{}&
C
\int_M
u^{-\frac{2}{n-4}}
\left(
|E_iu^i|+|F_iu^i|
\right)\eta^{\gamma-2}
+
C
\int_M
u^{-\frac{2}{n-4}-1}
(-h)|\nabla u|^2\eta^{\gamma-2}
\nonumber\\
&+
\frac12
\int_M
u^{-\frac{2}{n-4}}h^2\eta^{\gamma-2}
+
CR^{-2}
\int_M
u^{-\frac{2}{n-4}}
|\nabla u|^2\eta^{\gamma-4}.
\end{align*}
Hence
\begin{align*}
\int_M
u^{-\frac{2}{n-4}}h^2\eta^{\gamma-2}
\leqslant{}&
C
\int_M
u^{-\frac{2}{n-4}}
\left(
|E_iu^i|+|F_iu^i|
\right)\eta^{\gamma-2}
+
C
\int_M
u^{-\frac{2}{n-4}-1}
(-h)|\nabla u|^2\eta^{\gamma-2}
\nonumber\\
&+
CR^{-2}
\int_M
u^{-\frac{2}{n-4}}
|\nabla u|^2\eta^{\gamma-4}.
\tag{4.8}
\label{n>=7 2}
\end{align*}
Combining Proposition 2.2, \eqref{n>=7 1} and \eqref{n>=7 2}, we obtain
\begin{align*}
&\int_M
u^{-\frac{2}{n-4}}
\left[
h^2+\frac{|\nabla u|^4}{u^2}+(\Delta u)^2
\right]\eta^{\gamma-2}
\nonumber\\
&\leqslant{}
C
\int_M
u^{-\frac{2}{n-4}}
\left(
|E_iu^i|+|F_iu^i|
\right)\eta^{\gamma-2}
+
C
\int_M
u^{-\frac{2}{n-4}-1}
(-h)|\nabla u|^2\eta^{\gamma-2}
\nonumber\\
&
\quad
+
CR^{-2}
\int_M
u^{-\frac{2}{n-4}}
|\nabla u|^2\eta^{\gamma-4}.
\end{align*}
From \eqref{G 1},
\begin{equation*}
    -G
    \geqslant
    \frac1{c_n^2}\frac{\tilde h^2}{u}.
    \tag{4.9}
    \label{G 2}
\end{equation*}
Consequently, using Young's inequality,
\begin{align*}
&R^{-2}
\int_M
u^{-\frac{2}{n-4}}
\left[
h^2+\frac{|\nabla u|^4}{u^2}+(\Delta u)^2
\right]\eta^{\gamma-2}
\nonumber\\
&\leqslant{}
\varepsilon\Delta
+
CR^{-4}
\int_M
u^{-\frac{2}{n-4}}
|\nabla u|^2\eta^{\gamma-4}
+
CR^{-2}
\int_M
u^{-\frac{2}{n-4}-1}
(-\tilde h)|\nabla u|^2\eta^{\gamma-2}
\nonumber\\
&
\quad
+
CR^{-2}
\int_M
u^{\frac{2}{n-4}}
|\nabla u|^2\eta^{\gamma-2}.
\end{align*}
Using \eqref{G 2} and Young's inequality,
\begin{align*}
&R^{-2}
\int_M
u^{-\frac{2}{n-4}}
\left[
h^2+\frac{|\nabla u|^4}{u^2}+(\Delta u)^2
\right]\eta^{\gamma-2}
\nonumber\\
&\leqslant{}
\varepsilon\Delta
+
CR^{-4}
\int_M
u^{-\frac{2}{n-4}}
|\nabla u|^2\eta^{\gamma-4}
+
\frac{\varepsilon}{c_n^2}
\int_M
u^{-\frac{2}{n-4}-2}
\tilde h^2|\nabla u|^2\eta^\gamma
+
CR^{-2}
\int_M
u^{\frac{2}{n-4}}
|\nabla u|^2\eta^{\gamma-2}
\nonumber\\
&
\leqslant
2\varepsilon\Delta
+
CR^{-4}
\int_M
u^{-\frac{2}{n-4}}
|\nabla u|^2\eta^{\gamma-4}
+
CR^{-2}
\int_M
u^{\frac{2}{n-4}}
|\nabla u|^2\eta^{\gamma-2}
\nonumber\\
&
\leqslant{}
2\varepsilon\Delta
+
\frac12R^{-2}
\int_M
u^{-\frac{2}{n-4}}
\frac{|\nabla u|^4}{u^2}
\eta^{\gamma-2}
+
CR^{-6}
\int_M
u^{-\frac{2}{n-4}+2}\eta^{\gamma-6}
+
CR^{-2}
\int_M
u^{\frac{2n-2}{n-4}}\eta^{\gamma-2}.
\end{align*}
Thus
\begin{align*}
R^{-2}
\int_M
u^{-\frac{2}{n-4}}
(\Delta u)^2
\eta^{\gamma-2}
\leqslant{}&
4\varepsilon\Delta
+
CR^{-6}
\int_M
u^{-\frac{2}{n-4}+2}\eta^{\gamma-6}
+
CR^{-2}
\int_M
u^{\frac{2n-2}{n-4}}\eta^{\gamma-2}.
\end{align*}
Hence, choosing $\varepsilon\in\mathbb{R}_+$ sufficiently small and combining \eqref{Delta} give
\begin{align*}
\Delta
\leqslant{}&
C_0R^{-2}
\int_M
u^{-\frac{2}{n-4}}
(\Delta u)^2\eta^{\gamma-2}
\nonumber\\
\leqslant{}&
\frac12\Delta
+
CR^{-6}
\int_M
u^{-\frac{2}{n-4}+2}\eta^{\gamma-6}
+
CR^{-2}
\int_M
u^{\frac{2n-2}{n-4}}\eta^{\gamma-2}.
\end{align*}
Therefore,
\begin{equation*}
    \Delta
    \leqslant
    CR^{-6}
    \int_M
    u^{-\frac{2}{n-4}+2}\eta^{\gamma-6}
    +
    CR^{-2}
    \int_M
    u^{\frac{2n-2}{n-4}}\eta^{\gamma-2}.
    \tag{4.10}
    \label{n>=7 3}
\end{equation*}
The same estimate can also be obtained by using the argument for $n=6$ in the subsection 4.3.
Applying \eqref{vol} and Hölder's inequality,
\begin{align*}
\Delta
\leqslant{}&
CR^{-1}
\left(
\int_Mu^{\frac{2n}{n-4}}
\right)^{\frac{n-5}{n}}
+
CR^{-1}
\left(
\int_Mu^{\frac{2n}{n-4}}
\right)^{\frac{n-1}{n}}.
\end{align*}
Since $u\in \mathcal{L}^{\frac{2n}{n-4}}(M)$, letting $R\to+\infty$ gives $E_{ij}\equiv0$ and $F_i\equiv0$.  Furthermore, by the Caccioppoli trick, it suffices to assume that
\[
    \int_{B_R}
    u^{\frac{2n}{n-4}}
    =
    O\left(
    R^{\frac{n}{n-1}}
    \right),
\]
for some $p\in M$ and large $R$.

Moreover, when $n\geqslant7$, using \eqref{n+4/n-4},
\begin{align*}
    R^{-2}
    \int_M
    u^{\frac{2n-2}{n-4}}\eta^{\gamma-2}
    \leqslant{}&
    CR^{\frac n2-4}
    \|u\|_{\mathcal{L}^\infty(B_{2R})}^{\frac{n-6}{n-4}}.
\end{align*}
Thus using \eqref{n>=7 3} and the Caccioppoli trick, the pointwise growth or decay condition
\begin{equation*}
    \|u\|_{\mathcal{L}^\infty(B_R)}=O(R^{\mu_n}),
    \qquad
    \mu_n=\frac{(8-n)(n-4)}{2(n-6)},
    \tag{4.11}
    \label{growth 2}
\end{equation*}
for some $p\in M$ and large R, is also sufficient for the classification. In particular, from \eqref{growth 1} and \eqref{growth 2}, since $\mu_7=\frac32\in\mathbb{R}_+$ and $\mu_8=0$, we obtain the following corollary:
\begin{corollary}
Let $(M^n,g)$ be a complete, connected, non-compact Riemannian manifold,
$6\leqslant n\leqslant8$, satisfying $\mathrm{Ric}\geqslant 0$. Assume $f(u)=u^{\frac{n+4}{n-4}}$. Then there exists $u\in \mathcal{C}_b^\infty(M,\mathbb R_+)$ satisfying \eqref{eq} and \eqref{Delta u} if and only if $(M^n,g)$ is isometric to
$(\mathbb{R}^n,g_E)$. Moreover, $u$ is of the form
\begin{equation*}
    u(x)
    =
    \frac{
    \left[\lambda^2 n(n-4)(n^2-4)\right]^{\frac{n-4}{8}}
    }
    {\left(1+\lambda|x-x_0|^2\right)^{\frac{n-4}{2}}},
    \qquad
    \lambda\in\mathbb{R}_+,\quad x_0\in\mathbb{R}^n.
\end{equation*}
\end{corollary}


\begin{thebibliography}{99}\small
    \bibitem{CGS89} L. A. Caffarelli, B. Gidas, J. Spruck, Asymptotic symmetry and local behavior of semilinear elliptic equations with critical Sobolev growth, Comm. Pure Appl. Math. 42 (1989), 271--297.
    
    \bibitem{Cas24} J. Case, The Obata--V\'etois argument and its applications, J. Reine Angew. Math. (Crelles Journal) 815 (2024), 23--40.

    \bibitem{CM22} G. Catino, D. D. Monticelli, Semilinear elliptic equations on manifolds with nonnegative ricci curvature, arXiv: 2203.03345, 33 pp.
    
    \bibitem{CMR23} G. Catino, D. D. Monticelli, A. Roncoroni, On the critical \(p\)-Laplace equation, Adv. Math. 433 (2023), 38 pp.
    
    \bibitem{CL91} W. Chen, C. Li, Classification of solutions of some nonlinear elliptic equations, Duke Math. J. 63 (1991) 615--622.

    \bibitem{FWX15} M. Fazly, J. Wei and X. Xu, A pointwise inequality for the fourth-order Lane-Emden equation, Anal. PDE 8 (2015), 1541--1563.

    \bibitem{FMM23} M. Fogagnolo, A. Malchiodi, L. Mazzieri, A note on the critical Laplace equation and Ricci curvature, J. Geom. Anal. 33 (2023), 17 pp.
    
    \bibitem{GS81} B. Gidas, J. Spruck, Global and local behavior of positive solutions of nonlinear elliptic equations, Commun. Pure Appl. Math. 34 (1981), 525--598.
    
    \bibitem{LW25} M. Li, J. Wei, A remark on the Case-Gursky-V\'etois identity and its applications, Proc. Am. Math. Soc. 153 (2025), 3417--3430.
    
    \bibitem{Lin98} C. Lin, A classification of solutions of a conformally invariant fourth order equation in \(\mathbb{R}^n\), Comment. Math. Helv. 73 (1998), 206--231.
    
    \bibitem{MOW25} X. Ma, Q. Ou, T. Wu, Jerison-Lee identities and semi-linear subelliptic equations on Heisenberg group, Acta Math. Sci. 45 (2025), 264--279.
    
    \bibitem{MW24} X. Ma, T. Wu, The application of the invariant tensor technique in the classification of solutions to semilinear elliptic and sub-elliptic partial differential equations (in Chinese), Sci. Sin. Math. 54 (2024), 1627--1648.

    \bibitem{MWW25} X. Ma, T. Wu, W. Wu, Liouville theorem of the subcritical biharmonic equation on complete manifolds, arXiv:2508.14497, 13 pp.
    
    \bibitem{MWZ25a} X. Ma, T. Wu, X. Zhou, The Liouville-type equation and an Onofri-type inequality on closed 4-manifolds, arXiv:2508.14494, 17 pp.
    
    \bibitem{MWZ25b} X. Ma, T. Wu, X. Zhou, The subcritical biharmonic equation and Sobolev embedding nequality on closed 4-manifolds, preprint, (2025).
    
    \bibitem{NNP18} Q.~A. Ng\^o, V.~H. Nguyen and Q.~H. Phan, A pointwise inequality for a biharmonic equation with negative exponent and related problems, Nonlinearity 31 (2018), 5484--5499.

    \bibitem{Ou25} Q. Ou, On the classification of entire solutions to the critical \(p\)-Laplace equation, Math. Ann. 392 (2025), 1711--1729.

    \bibitem{SW25} L. Sun, Y. Wang, Critical quasilinear equations on Riemannian manifolds, arXiv: 2502.08495, 46 pp.

    \bibitem{Tas65} Y. Tashiro, Complete Riemannian manifolds and some vector fields, Trans. Amer. Math. Soc. 117 (1965), 251--275.
    
    \bibitem{Vet24} J. V\'etois, Uniqueness of conformal metrics with constant Q-curvature on closed Einstein manifolds, Potential Anal. 61 (2024), 485--500.
\end{thebibliography}
\end{document}